\documentclass[11pt]{article}
\usepackage{geometry} 
\usepackage{amsmath,amsfonts,euscript,amssymb,amsthm,times}
\usepackage{fancyhdr,graphicx,color,pstricks,pst-grad, multicol} 
\usepackage{setspace}
\usepackage{graphicx}
\usepackage{amssymb}
\usepackage{epstopdf}
\newcommand{\C}{{\cal C}}
\newcommand{\N}{{\cal N}}
\newcommand{\A}{{\cal{A}}}
\newcommand{\E}{{\cal{E}}}
\newcommand{\D}{{\cal{D}}}
\newcommand{\U}{{\cal{U}}}
\newcommand{\R}{\mathbb{R}}
\newtheorem{definition}{Definition}
\newtheorem{theorem}{Theorem}
\newtheorem{lemma}{Lemma}
\newtheorem{proposition}{Proposition}
\newtheorem{corollary}{Corollary}

\title{
Solutions of a pure critical exponent problem involving  the half-laplacian in annular-shaped domains}
\author{Antonio 
Capella\footnote{The author was supported by grants  PAPIIT  IN101209,  MTM2008-06349-C03-01 and 2009SGR345.}}
\begin{document}
\maketitle
\begin{center}
{\small 
Instituto de Matem\'aticas, Universidad Nacional Aut\'onoma de M\'exico, \\
Circuito Exterior, C.U., 04510 M\'exico D.F., Mexico. \\
E-mail address : capella@matem.unam.mx 
}
\end{center}
\medskip

\begin{abstract}
We consider the nonlinear and nonlocal problem 
$$
A_{1/2}u=|u|^{2^\sharp-2}u\ \text{ in }\ \Omega, \quad u=0 \text{ on } \partial\Omega
$$where  $A_{1/2}$ represents the square root of the Laplacian in a bounded domain
with zero Dirichlet boundary conditions, $\Omega$ is a bounded smooth domain in $\R^n$, $n\ge 2$
and $2^{\sharp}=2n/(n-1)$ is the critical trace-Sobolev exponent. We assume that $\Omega$ is 
annular-shaped, i.e., there exist $R_2>R_1>0$ constants such that 
$\{x\in\R^n\ \text{ s.t. }\ R_1<|x|<R_2\}\subset\Omega$ and $0\notin\Omega$, and 
invariant under a group $\Gamma$ of orthogonal transformations of $\R^n$ without fixed points. 
We establish the existence of positive and multiple sing changing solutions 
in the two following cases:  
if $R_1/R_2$ is arbitrary and the minimal $\Gamma$-orbit of $\Omega$ is large enough, 
or  if  $R_1/R_2$ is small enough and $\Gamma$ is arbitrary.
  \end{abstract}
\section{Introduction}

We are interested on the existence of solutions to the problem 
\begin{equation}\label{half}
\left\{
\begin{array}{ll}
A_{1/2} u = |u|^{2^\sharp-2}u&\text{in }  \Omega\\
\\
u=0&\text{on } \partial \Omega,
\end{array}
\right.
\end{equation}
where $\Omega$ is a smooth bounded domain in $\mathbb{R}^n$, $n\ge 2$ and $2^\sharp=2n/(n-1)$ is the critical trace-Sobolev exponent, and $A_{1/2}$ stands for the square root of the  Laplacian  $-\Delta$ in $\Omega$ 
with zero boundary values on $\Omega$. The action of $A_{1/2}$ can be defined as follows: let 
$\{\lambda_k, \varphi_k\}_{k=1}^\infty$ denote the eigenvalues and eigenfunctions of the Laplacian with zero 
Dirichlet boundary values on $\partial\Omega$.  Assume $\Vert \varphi_k\Vert_{L^2(\R^n)}=1$. Then 
the square root of the Dirichlet Laplacian, denoted by $A_{1/2}:H^1(\Omega)\to L^2(\Omega)$, is given by 
$$
u=\sum_{k=1}^\infty c_k\varphi_k\ \mapsto  A_{1/2}u=\sum_{k=1}^\infty c_k\lambda_k^{1/2}\varphi_k.
$$

Operators like  $A_{1/2}$, and in general fractions of the Laplacian, are infinitesimal generators of L\'evy stable  diffusion 
processes and appear in anomalous diffusion in fluids, flame propagation, chemical reactions in liquids, geophysical 
fluid dynamics and american options. 

The local version of problem \eqref{half} that involves the Laplacian operator and the critical exponent  
is given by 
\begin{equation}\label{lap}
\left\{
\begin{array}{ll}
-\Delta u = |u|^{2^*-2}u&\text{in }  \Omega\\
\\
u=0&\text{on } \partial \Omega,
\end{array}
\right.
\end{equation}
where $2^*:=2n/(n-2)$ is the critical Sobolev exponent. Problem \eqref{lap} has been widely studied. 
It is know that the existence of solutions depends on the domain $\Omega$. If $\Omega$ 
is strictly star-shaped, it was showed by Pohozaev  \cite{Pohozaev65} that  \eqref{lap} has  no non-trivial solution.  
In  \cite{BrezisNirenberg83},  Brezis and Nirenberg showed that by adding a small 
linear perturbation to the critical power nonlinearity compactness  and existence of solution are both restored.
The first existence results for nontrivial solutions was given by Coron  \cite{Coron84} for a domain $\Omega$ that
has a small enough hole. Later Bahri and Coron \cite{BahriCoron88} showed that the same holds for holes of any size.  
In \cite{ClappWeth09}, Clapp and Weth extended Coron's result to show that if $\Omega$ has a small enough hole, \eqref{lap} has at least two solutions. Regarding sign changing solutions, existence is only know 
for domains with symmetries. The first of these type of results  was given by Marchi and Pacella \cite{MachiPacella93}. 
For symmetric domains with small holes existence was showed by Clapp and Weth    \cite{ClappWeth04},   and 
Clapp and Pacella \cite{ClappPacella08}. In \cite{ClappPacella08} they assume that $\Omega$ is annular-shaped, i.e.,
$$
0\notin \Omega\quad \text{and}\quad\Omega  \supset A_{R_1,R_2}=\left\{ x\in\R^n\ \text{ such that }\ R_1<|x|<R_2\right\}
$$
for some $0<R_1<R_2$ and  invariant under the action of a group $\Gamma$ of orthogonal transformations of $\R^n$.
Under these assumptions Clapp and Pacella showed existence in the following two cases:  
if $R_1/R_2$ is arbitrary and the minimal $\Gamma$-orbit of $\Omega$ is large enough, 
or  if  $R_1/R_2$ is small enough and $\Gamma$ is arbitrary.
The multiplicity results of \cite{ClappPacella08} are obtained by using the  invariance of \eqref{lap} under 
the group of M\"obius transformations (see section 2 in \cite{ClappPacella08}).

For the square root $A_{1/2}$ of the Laplacian, Tan established in \cite{Tan09} the nonexistence of 
classical solutions to \eqref{half} for star-shaped domains. In addition,  Tan also showed  in \cite{Tan09} 
a  Brezis-Nirenberg type result for nonlinearities of the form $f(u)=|u|^{2^\sharp-2}u+\mu u$, $\mu>0$.
In \cite{CabreTan09}, Cabre and Tan studied  existence, regularity  and symmetry results for problem 
\eqref{half} with power nonlinearities. In particular they showed  a nonexistence Liouville result for  
{\it bounded} solutions of  \eqref{half} in  $\R^n$, or in $\R^n_+$. 
The corresponding result for unbounded solutions, that is know for \eqref{lap},  is still open. 

The aim of the present paper is to prove existence of solutions for \eqref{half}. To this end we apply, 
adapted to our context,  the variational principle for sign changing solutions developed in  \cite{ClappPacella08}.
Using this method we  obtain multiplicity results for \eqref{half} similar to the ones given by  Clapp and Pacella.
This variational principle is based on standard variational methods combined with symmetry  
assumptions to increase the energy interval in which the Palais-Smale condition holds. 
At the core of the argument are some model functions in low energy finite dimensional spaces, 
which allow to produce multiple solutions.  

Now, we state our hypothesis on the domain. 
As in \cite{ClappPacella08}, we assume that  $\Omega$ is annular-shaped and invariant under the action of a closed subgroup  $\Gamma$ of $O(n)$, that is, of the orthogonal transformations of $\R^n$. We denote by $\Gamma x:=\{\gamma x\ :\  \gamma\in \Gamma\}$ the  $\Gamma$-orbit of $x\in\R^n$, by $\text{\#}\Gamma x$ its cardinality, and 
let   
$$
\ell=\ell(\Gamma):=\min\{\text{\#}\Gamma\ \text{ such that }\ x\in\R^n\setminus\{0\}\}.
$$
We say that $\Omega$ is $\Gamma$-invariant if $\Gamma x\subset \Omega$ for every $x\in\Omega$.
A function $u:\Omega\to\R$ is $\Gamma$-invariant if it is constant along every $\Gamma$-orbit. 

Now  we state our main results:

\begin{theorem} \label{main_half1}
Given $0<R_1<R_2$ and $m\in\mathbb{N}$, there exist a positive integer 
$\ell_0$, depending on $m$ and $R_2/R_1$ such that,  for every closed subgroup $\Gamma$ of 
$O(n)$ with $\ell(\Gamma)>\ell_0$ and every $\Gamma$-invariant domain $\Omega$ with 
$$
0\notin \Omega\quad\text{and}\quad \{x\in \R^n \ : \ R_1<|x|< R_2\}\subset \Omega,
$$
problem \eqref{half} has at least one positive $\Gamma$-invariant solution $u_1$ and $m-1$ distinct 
 pairs of $\Gamma$-invariant sign changing solutions $\pm u_2,\dots,\pm u_m$.
\end{theorem}

In our second main result we consider the existence of solutions for domains with a small  hole. 

\begin{theorem}\label{main_half2}
Given $\delta>0$ there exist $R_\delta$ such that: for every closed subgroup $\Gamma$ of $O(n)$ with 
$\ell=\ell(\Gamma)\ge 2$ and every $\Gamma$-invariant domain $\Omega$ such that  
$$
0\notin \Omega\quad\text{and}\quad \{x\in \R^n \ : \ R_1<|x|< R_2\}\subset \Omega
$$
and $$
0<R_1/R_2 <R_\delta,
$$
problem \eqref{half} has at least one positive $\Gamma$-invariant solution $u_1$ and $\ell$ pairs of 
distinct $\Gamma$-invariant sign changing solutions $\pm u_2,\dots,\pm u_{\ell+1}$.
\end{theorem}

The proofs of Theorem~\ref{main_half1} and Theorem~\ref{main_half2} are based on the following 
result  proved by Cabre and Tan \cite{CabreTan09}:  the nonlocal  problem \eqref{half} 
can be realized through a local problem in one more dimension.
More precisely, denote the half-cylinder 
$$\C=\Omega\times(0,\infty)$$  
and its lateral boundary 
$$\partial_L\C=\partial\Omega\times(0,\infty).$$
Then, if $u$ is a function defined in $\Omega$, consider its harmonic extension $v$  in $\C$ 
with $v$ vanishing on $\partial_L\C$, then $A_{1/2}$ is given by the Dirichlet to Newmann map 
on $\Omega$,  $u\mapsto\frac{\partial u}{\partial\nu}\big|_{\Omega\times\{0\}}$ of such harmonic 
extension over the cylinder $\C$.  
Therefore, instead of \eqref{half} we are lead to consider the 
following mixed boundary value problem  
\begin{equation}\label{ext}
\left\{
\begin{array}{lll}
-\Delta v = 0&\text{in }& \C,\\
v=0&\text{on }&\partial_L\C, \\
\frac{\partial v}{\partial\nu}  =  |u|^{2^\sharp-2}u &\text{on } &\Omega\times\{0\},
\end{array}
\right.
\end{equation}
where $\nu$ is the unit outher normal to $\Omega\times\{0\}$. 
If $v$ satisfies \eqref{ext} then the trace $u$  on $\Omega\times\{0\}$ of  $v$  is a solution of \eqref{half}. 
As natural space for solutions of \eqref{ext} we consider 
$$
H^1_{0,L}(\C):=\left\{v\in H^1(\C)\ \text{ such that }\ v=0 \text{ a.e.  on }\ \partial_L\C\right\}.
$$
We denote by $Tr_\Omega$ the trace operator on $\Omega\times\{0\}$ for functions in $H^1_{0,L}(\C)$, and 
consistently use the notation
$$
u=Tr_\Omega(v)\quad\text{for }\  H^1_{0,L}(\C).
$$
For further details on this representation and the involved functional spaces we refer to \cite{CabreTan09}. 

Now we have to explain how we translate our assumptions on the domain $\Omega$ to the extended domain $\C$. 
First,  we say that $\C$ is an annular-shaped cylinder, if 
$$
0\notin \C\quad\text{and}\quad \C\supset {\cal A}_{R_1,R_2}:=\{(x,y)\in \R^{n+1}_+ \ : \ R_1<|x|< R_2, y\in[0,\infty)\}
$$
for some $0<R_1<R_2$. Second, we say that a closed subgroup $\Gamma$ of $O(n)$ acts on the base of  
$$\R^{n+1}_+:=\{(x,y)\subset \R^{n+1}\ \text{such that } \ x\in \R^n, x\in (0,\infty)\},$$ 
if  for $\gamma\in \Gamma$ and $(x,y)\in \R^{n+1}_+$ 
$$
\gamma(x, y)=(\gamma x, y).
$$
Hence,  $\Gamma (x,y):=\left(\{\gamma x\ \text{ s.t. }\  \gamma\in \Gamma\},y\right)$ denote its  $\Gamma$-orbit
and  $\text{\#}\Gamma (x,y)=\text{\#}\Gamma x$ its  cardinality. Finally,   we let
$$
\ell=\ell(\Gamma):=\min\{\text{\#}\Gamma\ \text{ s.t. }\  (x,y) \in\R^{n+1}_+\setminus\{0\}\}.
$$
We say  that $\C$ is  $\Gamma$-invariant if $\Gamma (x,y)\subset\C$ for every $(x,y)\in \C$, that is, 
$\C$ is invariant under $\Gamma$ acting on the base of the cylinder. As before, a function $v$ is said to be 
$\Gamma$-invariant if it is constant on every $\Gamma$-orbit.  

Thus, Theorem~\ref{main_half1} and Theorem~\ref{main_half2} are corollaries 
of the following results:

\begin{theorem} \label{main}
Given $0<R_1<R_2$ and $m\in\mathbb{N}$, there exist a positive integer 
$\ell_0$, depending on $m$ and $R_2/R_1$ such that,  for every closed subgroup $\Gamma$ of 
$O(n)$ with $\ell(\Gamma)>\ell_0$ and every $\Gamma$-invariant and annular-shaped cylinder 
domain $\C$ with 
$$
0\notin \C\quad\text{and}\quad \{(x,y)\in \R^{n+1}_+ \ : \ R_1<|x|< R_2, y\in[0,\infty)\}\subset \C,
$$
problem \eqref{ext} has at least one positive $\Gamma$-invariant solution $v_1$ and $m-1$ distinct 
 pairs of $\Gamma$-invariant sign changing solutions $\pm v_2,\dots,\pm v_m$.
\end{theorem}

Regarding the result of domains with a small hole we have:

\begin{theorem}\label{main2}
Given $\delta>0$ there exist $R_\delta$ such that: for every closed subgroup $\Gamma$ of $O(n)$ with 
$\ell=\ell(\Gamma)\ge 2$ and every $\Gamma$-invariant domain $\C$ such that  
$$
0\notin \C\quad\text{and}\quad \{(x,y)\in \R^{n+1}_+ \ : \ R_1<|x|< R_2, y\in[0,\infty)\}\subset \C
$$
and $$
0<R_1/R_2 <R_\delta,
$$
problem \eqref{ext} has at least one positive $\Gamma$-invariant solution $v_1$ and $\ell$ pairs of 
distinct $\Gamma$-invariant sign changing solutions $\pm v_2,\dots,\pm v_{\ell+1}$.
\end{theorem}

Observe that  the proofs given by Clapp and Pacella in  \cite{ClappPacella08}, 
for the multiplicity results corresponding to Theorems~\ref{main} and Theorem~\ref{main2},  relay on M\"obius invariance and in particular on inversion over spheres. In our context, due to the cylindrical shape of our domains, we can not use 
this inversion. Nevertheless, we manage to carry out the program of \cite{ClappPacella08} by replacing 
inversion over spheres and Kelvin transform by  dilations over cylinders and rescaling.

 The paper is organized as follows: in Section 2 we define the dilation invariance and construct some 
 radially symmetric test functions with controlled energy (Lemma~\ref{pos}).  In Section 3 we prove 
 the variational method for sign changing solutions of \cite{ClappPacella08} adaptaed to our context. 
 In Section 4, we prove two compactness lemmas and Theorem~\ref{main} and Theorem~\ref{main2}.
 Finally, in the appendix we give the proof of Struwe's lemma (Lemma 3.3 Chapter III in \cite{Struwe90})
 adapted to our context.

%%%%%%%%%%%%%%% Dilation invariance %%%%%%%%%%%%%%%%%%%%%%%

\section{Dilation invariance and group action}

Let $\lambda>0$ and  $\phi:\R_+^{n+1}\to \R_+^{n+1}$ be the dilation given by 
$\phi(x,y)=(\lambda x, \lambda y)$. This M\"obius transformation maps any cylinder
$\C$ into its rescaled version $\phi(\C)=\lambda\C$. For $v\in H^1_{0,L}(\phi(\C))$ , 
we define $v_\phi\in H^1_{0,L}(\C)$ by
\begin{equation}\label{action}
v_\phi(x,y):=(\det D\phi_x)^\frac{n-1}{2n}v(\phi(x,y)),
\end{equation} 
where $\det D\phi_x$ is the Jacobian determinant of the transformation restricted 
to the $x$ variables, that is  $\phi_x(x)=\lambda x$ and  $\det D\phi_x=\lambda^n$. 
The map $v\mapsto v_\phi$ is a linear isometry of $H^1_{0,L}(\phi(\C))\cong H^1_{0,L}(\C)$ 
and of $L^{2^\sharp}(\phi_x(\Omega))\cong L^{2^\sharp}(\Omega)$, i.e., 
\begin{equation}\label{inv}
\int_\C\nabla v_\phi\cdot \nabla \tilde{v}_\phi\ dxdy
=\int_{\phi(\C)}\nabla v\cdot \nabla \tilde v \ dxdy
\quad \text{and}\quad 
\int_{\Omega}|u_\phi|^{2^\sharp}dx= \int_{\phi_x(\Omega)}|u|^{2^\sharp}dx
\end{equation}

We write 
$$
\Vert v\Vert^2:=\int_\C|\nabla v(x,y)|^2dxdy\quad\text{and}\quad |u|^{2^\sharp}_{2^\sharp}:=\int_\Omega|u|^{2^\sharp}dx
$$
The solutions of \eqref{ext} are critical points of the energy functional $J_\Omega: H^1_{0,L}(\C)\to  \R$ defined by
\begin{equation}\label{enrg}
J_\Omega(v):=\frac{1}{2}\Vert v\Vert^2-\frac{1}{2^\sharp}|u|^{2^\sharp}_{2^\sharp}.
\end{equation}
From the invariance \eqref{inv} it follows that $J_{\Omega}(v_\phi)=J_{\phi(\Omega)}(v)$. This property allow us to construct certain sign changing test functions, that will be important in the proof of our main theorems. 

Let $G$ be a closed subgroup of $O(n)$, and assume that $\C$ is $G$-invariant on the base. 
Notice that in this case  the  orthogonal action of $G$ on $H^1_{0,L}(\C)$ is given 
by $\phi v: =v_{\phi^{-1}}$ for every $\phi\in G$ where $v_{\phi^{-1}}$ is defined as in  \eqref{action}.
Let 
$$
H^1_{0,L}(\C)^G:=\{ v\in H^1_{0,L}(\C)\  \text{ such that } \ \phi v = v \ \text { for all } \phi\in G\}
$$
be the subspace of  $H^1_{0,L}(\C)$ of  $G$-invariant functions.
Clearly,  the functional $J_\Omega$ is $G$-invariant, and by the principle of symmetric 
criticality \cite[Theorem 1.28]{Willem96}, the restriction of $J_\Omega$ to the space $H^1_{0,L}(\C)^G$ 
are solutions of \eqref{ext}. The nontrivial ones belong to the Nehari manifold 
$$
\N(\Omega)^G:=\left\{ v\in H^1_{0,L}(\C)^G\ \text{ such that }\ v\neq 0, \ \Vert v\Vert^2=|u|_{2^\sharp}^{2^\sharp}\right\}.
$$
If $G=\{Id\}$ is the trivial group, then $H^1_{0,L}(\C)^G=H^1_{0,L}(\C)$ and 
$\N(\Omega)^G=\N(\Omega)$ is the usual Nehari manifold. In this case, 
$$
\inf\left\{J_\Omega(v)\ \text{ such that }\ v\in \N(\Omega)\right\}=\frac{1}{2n}S_0^\frac{n-1}{2}=:c_\infty,
$$
where $S_0$ is the best constant in the Sobolev trace inequality \cite{PLLions85} for the embedding 
$\D^{1,2}(\R^{n+1}_+)\hookrightarrow L^{2^\sharp}(\R^n)$, where $\D^{1,2}(\R^{n+1}_+)$ denote the closure 
of the set of smooth functions compactly supported in $\overline{\mathbb{R}^{n+1}_+}$ with respect 
to the norm $\Vert w \Vert^2=\int_{\R^{n+1}_+}|\nabla w(x,y)|^2dxdy$. It is clear that this infimum does not depend 
on $\C$ and it is never attained for bounded $\Omega$. 
It was show in \cite{Escobar88} that this infimum is achieved by functions of the form 
$$
U_\varepsilon(x,y)=\frac{\varepsilon^{(n-1)/2}}{\left(|x-x_0|^2+(y+\varepsilon)^2\right)^\frac{n-1}{2}}, 
$$  
where $x_0\in \R^n$ and $\varepsilon>0$ is arbitrary. In addition the best constant is 
$$
S_0=\frac{(n-1)\sigma_n^{1/n}}{2},
$$
where $\sigma_n$ is the volume of the $n$-dimensional sphere in $\R^{n+1}$.
 If $G$ is nontrivial, the infimum for bounded $\Omega$ might be attained.  For example, 
 if 
 $$
\A=\A_{R_1,R_2}:=\{(x,y)\in \R^{n+1}_+\ : \ x\in A_{R_1,R_2} \ \text{ and } y\in [0,\infty)\}
$$
 where 
$$
A=A_{R_1,R_2}:=\{x\in \R^n\ : \ 0<R_1<|x|<R_2\},
$$
and  $G=O(n)$, then $H^1_{0,L}(\A_{R_1,R_2})$ becomes the space of radial functions and 
$$
c(R_1,R_2):=\inf\left\{J_{A_{R_1,R_2}}(v)\ \text{ such that }\ v\in \N(\A_{R_1,R_2})^{O(n)}\right\}
$$
is always attained.

The next elementary lemma contains  this result.

\begin{lemma}\label{radial}
Let $n\ge 2$, $0<R_1<R_2$ and $\C=\A_{R_1,R_2}$, there exist a radially symmetric, 
 positive, classical solution solution of \eqref{ext}.
\end{lemma}

\begin{proof}
Consider the energy functional
$$
J^+_A(v):=\frac{1}{2}\Vert v\Vert^2-\frac{1}{2^\sharp}|u^+|^{2^\sharp}_{2^\sharp}
$$ 
restricted to $X=\{v\in H^1_{0,L}(\A)\ : \ u=tr_\Omega (v) \text{ is radially symmetric}\}$ and where as usual 
$u^+:=\max\{0,u\}$.  It is straight forward that $J^+_A$  satisfies the assumptions of the mountain pass theorem. 
Thus,  it only remains to prove the Palais-Smale condition. Let $(v_m)_m\subset X$ be such that 
 $$
 d:=\sup_n J^+_A(v_m) < \infty \quad\text{and}\quad  D J^+_A(v_m)\to 0\quad X',
 $$ 
 where $X'$ denotes the dual space of $X$.
To show that $(v_m)_m$ is bounded, we compute 
\begin{eqnarray*}
o(1)(1+\Vert v_m\Vert)+2d & \ge & 2J_A^+(v_m)-\langle D J^+_A(v_m),v_m\rangle \\
& = & \left(1-\frac{2}{2^\sharp}\right)|u_m|_{2^\sharp}^{2^\sharp}
=\frac{1}{n}|u_m|_{2^\sharp}^{2^\sharp},
\end{eqnarray*}
where $o(1)\to 0$ as $n\to \infty$. Hence, 
\begin{eqnarray*}
\Vert v_m\Vert=2J^+_A(v_m)+\frac{2}{2^\sharp}|u_m|_{2^\sharp}^{2^\sharp}\le 4d+o(1)\Vert v_m\Vert^2
\end{eqnarray*}
and $(v_m)_m$ is bounded. 
Hence, we may assume
$$
v_m\rightharpoonup v\quad\text{in}\ H^1_{0,L}(\A).
$$ 
By the compactness of the embedding $Tr_\Omega (H^1_{0,L}(\A))\subset L^{p}(A)$ for every $1\le p<\infty$ 
in annular domains, we have
\begin{equation}\label{L2shrp_weak}
u_m=Tr_\Omega(v_m)\to u \quad \text{in}\ L^{2^\sharp}(A).
\end{equation}
In turn,  this imply that 
$$
|u_m|^{2^\sharp-2}u_m\to |u|^{2^\sharp-2}u\quad \text{in } L^{2^\sharp/(2^\sharp-1)}(A).
$$
Now, we observe that 
\begin{eqnarray*}
\int_\A|\nabla(v_m-v)|^2dxdy & = &\langle D J^+_A(v_m)-D J^+_A(v),v_m-v\rangle \\
 &&\hspace{2cm}
+\int_A\left(|u_m^+|^{2^\sharp-2}u_m^+-|u^+|^{2^\sharp-2}u^+\right)(u_m-u)\ dx.
\end{eqnarray*}
Clearly,  by weak convergence
$$
\langle D J^+_A(v_m)-D J^+_A(v),v_m-v\rangle\to 0, \quad \text{as } \ m\to \infty,
$$
and from H\"older inequality it follows  
\begin{eqnarray*}
\int_A\left(|u_m^+|^{2^\sharp-2}u_m^+-|u^+|^{2^\sharp-2}u\right)(u_m-u)\ dx&&\\
&&\hspace{-4cm}
\le \left(\int_A\big||u_m^+|^{2^\sharp-2}u_m^+-|u^+|^{2^\sharp-2}u\big|^qdx\right)^{1/q}
\left(\int_A|u-u_m|^{2^\sharp}dx\right)^{1/2^\sharp},
\end{eqnarray*}
where $q=2^\sharp/(2^\sharp-1)$.
Thus,  from \eqref{L2shrp_weak} by letting $m\to \infty$ we conclude that  $\Vert v_m-v\Vert\to 0$. 

Finally, by the regularity theory  (see \cite[Section 3]{CabreTan09})  we find that $v\in  C^{2,\alpha}(\A)$, and the maximum 
principle  (see also \cite[Section 4]{CabreTan09})  implies that $v$ is positive in $\A$.
\end{proof}

In the following lemma, we construct some radially symmetric test functions with controlled energy. 
These functions  will be used in the proof of our main theorems.

\begin{lemma}\label{pos}
Given $0<R_1<R_2$ and $m\in \mathbb{N}$, there exist $R_1=:P_0<P_1<\cdots<P_m:=R_2$ and 
positive radial functions $\omega_1, \dots,\omega_m\in \N(A_{R_1,R_2})^{O(n)}$ such that 
$$
\text{supp}(\omega_i)\subset \A_{P_i,P_{i+1}}\quad \text{ and } \quad J_A(\omega_i)=c(R_1^{1/m},R_2^{1/m}), 
\quad i=1,\dots,m.
$$ 
\end{lemma}

\begin{proof}
Let $\lambda=(R_2/R_1)^{1/m}$ and define $P_i=\lambda^iR_1$, for $i=1,\dots,m$. Let $\phi$ be the dilation 
by $\lambda$, that is 
$$
\phi(x,y)=\lambda (x,y).
$$
Now, fix a positive radial minimizer $\omega_1$ of $J_A$ on $\N(A_{P_0,P_1})^{O(n)}$ and define 
$$
\omega_{i+1}(x,y):=\lambda^\frac{n-1}{2}\omega_i(\lambda x, \lambda y).
$$

Since $\phi(\A_{P_{i-1},P_{i}})=\A_{P_i,P_{i+1}}$, the invariance by dilations yields that $w_{i+1}$ is a positive 
radial minimizer of $J_A$ on $\N(A_{P_{i},P_{i+1}})^{O(n)}$, with $J_{A_{P_i,P_{i+1}}}(\omega_{i+1})=J_{A_{P_0,P_1}}(\omega_1)=c(P_0,P_1)$. Finally, by rescaling, it follows easily  that $c(P_0,P_1)=c(R_1^{1/m},R_2^{1/m})$.
\end{proof}

%%%%%%%%%%%%%%% A variational principle %%%%%%%%%%%%%%

\section{A variational principle for sign changing solutions}

In this section we prove a mountain pass lemma for sign changing solutions. 
The results in this section closely follow the ones of \cite[Section 3]{ClappPacella08} adapted to the present  
setting. For completeness we will quote all need results, and where no mayor changes 
are needed we refer to the proof  given in \cite{ClappPacella08}.
\smallskip

Let $G$ be a closed subgroup of $O(n)$ and let $\C=\Omega\times[0,\infty)$ be an 
$G$-invariant on the base subset of $\R^{n+1}_+$. If $v$ is a sign changing $G$ -invariant 
solution of \eqref{ext} it must lie in the set
$$
\E^G:=\{v\in \N(\Omega)^G\ \text{ such that }\  v^+,v^-\in \N(\Omega)^G\}
$$
 where $v^+:=\max\{0,v\}$ and $v^-:=\min\{0,v\}$. Consider the negative gradient flow $\varphi:{\cal G}\to H^1_{0,L}(\C)^G$ of $J_\Omega$, defined by
 $$
 \left\{
 \begin{array}{l}
 \partial_t\varphi(t,v)=-D J_\Omega(\varphi(t,v)))\\
 \\
 \varphi(0,v)=v,
 \end{array}
 \right.
 $$
where ${\cal G}:=\{(t,v),\text{ s.t. } v\in H^1_{0,L}(\C)^G, 0\le t\le T(v)\}$ and $T(v)\in (0,\infty]$ is the maximal existence time of the trajectory $t\mapsto \varphi(t,v)$. We say that a subset $\D$ of $H^1_{0,L}(\C)^G$ is strictly
positive invariant under $\varphi$ if 
$$
\varphi(t,v)\in \text{int}(\D)\quad\text{for every } v\in \D\ \text{and every } t\in (0,T(v))
$$ 
where $\text{int}(\D)$ denotes the interior of $\D$ in $H^1_{0,L}(\C)^G$. If $\D$ is strictly positively invariant under 
$\varphi$, then the set 
$$
{\cal J}(\D):=\{ v\in H^1_{0,L}(\C)^G\text {s.t. } \varphi(t,v)\in \D\text{ for some } t\in (0,T(v))\}
$$
is open in $H^1_{0,L}(\C)^G$, and the time entrance map $\tau_\D:{\cal J}(\D)\to \R$ defined by 
$$
\tau_\D(v):=\inf\{ t\ge 0\text{ s.t. } \varphi(t,v)\in \D\}
$$ 
is continuos. We write ${\cal P}^G:=\{v\in H^1_{0,L}(\C)^G\text{ s.t. } v\ge 0\}$ for the convex cone of positive functions in $H^1_{0,L}(\C)^G$ and, for  $\alpha>0$, we set 
$$
B_\alpha({\cal P}^G):=\{ v\in H^1_{0,L}(\C)\text{ s.t. } \text{dist}(v,{\cal P}^G)\le\alpha\},
$$ 
where $\text{dist}(v,{\cal J}):=\inf_{w\in{\cal J}}\Vert v-w\Vert$.

\begin{lemma}\label{lem:alpha}
There exists $\alpha>0$ such that\\
(a) \ $\left[B_\alpha({\cal P}^G)\cup B_\alpha(-{\cal P}^G)\right]\cap \E^G =\emptyset$ \\
(b) \  $B_\alpha({\cal P}^G)$ and $B_\alpha(-{\cal P}^G)$ are strictly invariant under $\varphi$.
\end{lemma}

\begin{proof}
Now we give the argument for (a). For every $v\in H^1_{0,L}(\C)^G$, 
\begin{equation}\label{distneg}
|u^-|_{2^\sharp}=\min_{w\in {\cal P}^G}|u-Tr_\Omega(w)|_{2^\sharp}
\le S_0^{-1}\min_{w\in {\cal P}^G}\Vert v-w\Vert=S_0^{-1}\text{dist}(v,{\cal P}^G),
\end{equation}
where $S_0$ is the best constant in the trace-Sobolev embedding $H^1_{0,L}(\R^{n+1}_+) \hookrightarrow L^{2^\sharp}(\R^n)$. Therefore, since $\inf_{\N(\Omega)^G} J_\Omega > 0$, there exist $\alpha$ such that $\text{dist}(v,{\cal P}^G) > \alpha$ for every $v\in \E^G$. Moreover, since $\E^G$ is symmetric with respect to the origin, 
$\text{dist}(v,-{\cal P}^G)=\text{dist}(-v,{\cal P}^G)>\alpha$, and (a) follows.

In order to prove (b), we only need to consider $B_\alpha({\cal P}^G )$. 
The gradient  $DJ_\Omega:H^1_{0,L}(\C)^G\to H^1_{0,L}(\C)^G$ is given by 
$$
DJ_\Omega(v)=v-K(u),
$$
where $K(u)$ is the unique solution to 
$$
\left\{\begin{array}{ll}
-\Delta K(u)=0 & \text{in } \C\\
K(u)=0&\text{on } \partial_L\C\\
\frac{\partial}{\partial\nu} K(u)=|u|^{2^\sharp-2}u&\text{on } \Omega\times\{0\}.
\end{array}
\right.
$$
That is, $K(u)$ is determined by the relation 
$$
\langle K(u),w\rangle = \int_\Omega  |u|^{2^\sharp-2}u\ Tr_\Omega(w)\ dx 
$$
for every $w\in H^1_{0,L}(\C)$. By the maximum principle \cite[Lemma 4.1]{CabreTan09}  
$$
K(u)\in {\cal P}^G\ \text{ if }\ v\in {\cal P}^G, 
$$ 
we recall that $Tr_\Omega(v)=u\ge 0$  for every $v\in {\cal P}^G$.
Let $v\in H^1_{0,L}(\C)^G$ and $w\in {\cal P}^G$ be such that 
$\text{dist}(v,{\cal P}^G)=\Vert v-w\Vert$. 
We find
\begin{eqnarray*}
\text{dist}(K(u),{\cal P}^G) \Vert K(u)^-\Vert \le  \Vert K(u)^-\Vert^2
&=&  \langle K(u),K(u)^-\rangle \\
&= & \int_\Omega  |u|^{2^\sharp-2}u\ Tr_\Omega(K(u)^-)\ dx \\
& \le & \int_\Omega  |u^-|^{2^\sharp-2}u^-\ Tr_\Omega(K(u)^-)\ dx \\
& \le & |u^-|_{2^\sharp}^{2^\sharp-1}|Tr_\Omega(K(u)^-)|_{2^\sharp}\\
&\le& S_0^{2^\sharp-1}\text{dist}(v,{\cal P}^G)^{2^\sharp-1} \Vert K(u)^-\Vert.
\end{eqnarray*}
Hence, 
$$
\text{dist}(K(u),{\cal P}^G) \le S_0^{2^\sharp-1}\text{dist}(v,{\cal P}^G)^{2^\sharp-1}.
$$
Then, given $\nu <1$ there exist an $\alpha_0>0$ such that if $\alpha<\alpha_0$, 
$$
\text{dist}(K(u),{\cal P}^G) \le \nu\  \text{dist}(v,{\cal P}^G)\quad\text{for every} \ v\in B_\alpha({\cal P}^G).
$$
Thus, $K(u)\in \text{int}(B_\alpha({\cal P}^G))$ if $v\in B_\alpha({\cal P}^G)$. 
Since $B_\alpha({\cal P}^G)$ is closed and convex,  Theorem 5.2 in \cite{Deimling77}
implies 
\begin{equation}\label{flux}
v\in B_\alpha({\cal P}^G) \quad \Rightarrow \quad \varphi(t,v)\in B_\alpha({\cal P}^G)\ \text{ for }\ t\in [0, T(v)).\end{equation}
To conclude the proof,  by contradiction we assume that 
there exist $v \in B_\alpha({\cal P}^G)$ and $t\in (0,T(v))$ such that 
$\varphi(t,v)\in \partial B_\alpha({\cal P}^G)$. Mazur's separation theorem 
(see e.g. Theorem 2.219 in \cite{Megginson98}) gives 
the existence of a continuos linear functional $L\in (H^1_{0,L}(\C)^G)'$ and $\beta>0$ 
such that $L(\varphi(t,v))=\beta$ and $L(v) > \beta$. for $v\in \text{int}( B_\alpha({\cal P}^G))$.
It follows
$$
\left. \frac{\partial}{\partial s}\right|_{s=t}L(\varphi(s,v))=L(-DJ(\varphi(t,v)))=L(K(\varphi(t,v)))-\beta >0.
$$
Hence, there exists $\varepsilon>0$ such that $L(\varphi(s,v)) < \beta$ for $s\in (t-\varepsilon,t)$. 
 Thus, $\varphi(s,v)\notin  B_\alpha({\cal P}^G)$ for $s\in (t-\varepsilon, t)$. This contradicts \eqref{flux} and 
 finish the proof. 
\end{proof}

Fix $\alpha>0$ as in Lemma~\ref{lem:alpha}. Then $J_\Omega$ has no sign changing critical points in $B_\alpha({\cal P}^G)\cup B_\alpha(-{\cal P}^G)$. Let $J^d:=\{v\in H^1_{0,L}(\C)^G \text{ s.t. } J_\Omega(v)\le d\}$.

\begin{corollary} If $J_\Omega$ has no changing critical points $v\in H^1_{0,L}(\C)^G$ with $J_\Omega(v)=d$, then 
the set 
$$
\D_d^G:=B_\alpha({\cal P}^G)\cup B_\alpha(-{\cal P}^G)\cup J^d
$$
is strictly positively invariant under $\varphi$, and the map
$$
\varrho:{\cal J}(\D^G_d)\to \D^G_d, \quad \varrho_d(v):=\varphi(e_{\D^G_d}(v),v)
$$ 
is odd and continuous, and satisfies $\varrho_d(v)=v$ for every $v\in \D^G_d$. 
\end{corollary}

We will say that a subset ${\cal Y}$ of $H^1_{0,L}(\C)^G$ is symmetric if $-v\in {\cal Y}$ for every $v\in{\cal Y}$.

\begin{definition}
Let $\D$ and ${\cal Y}$ be symmetric subsets of $H^1_{0,L}(\C)^G$. The  genus $g({\cal Y},\D)$ of ${\cal Y}$ relative to $\D$ is defined as the smallest number $m$ such that ${\cal Y}$ can be covered by $m+1$ open symmetric subset $\U_0,\U_1,\dots,\U_m$ of  $H^1_{0,L}(\C)^G$ such that: 
\begin{enumerate}
\item[(i)] ${\cal Y}\cap \D\subset \U_0$ and there exists an odd continuous map $\vartheta_0:\U_0\to \D$ such that 
$\vartheta_0(v)=v$ for $v\in {\cal Y}\cap \D$.
\item[(ii)] there exist odd continuous maps $\vartheta_j:\U_j\to \{1,-1\}$ for every $j=1,\dots,m$. 
\end{enumerate}
If no such cover exists, we define $g({\cal Y},\D)=\infty$.
\end{definition}

If $\D=\emptyset$ we write $g({\cal Y})=g({\cal Y}, \emptyset)$ and as pointed in \cite{ClappPacella08} this is the usual Krasnoselskii genus. 
The set $\D$ is called a symmetric neighborhood retract if there exist a symmetric neighborhood $\U$ of $\D$
in $H^1_{0,L}(\C)^G$ and an odd continuous map $\varrho:\U\to \D$ such that $\varrho(v)=v$ for every  
$v\in \D$.

\begin{definition}
Let $\D\subset{\cal H}$ be subsets of $H^1_{0,L}(\C)$. We say that  $J_\Omega$ satisfies $(PS)_c$ relative to $\D$ in ${\cal H}$, if every sequence $(v_m)_m$ in ${\cal H}$ such that 
$$v_m\notin \D,\quad J_\Omega(v_m)\to c, \quad D J_\Omega(v_m)\to 0,$$
has a convergent subsequence. If $\D=\emptyset$ we simply say that $J_\Omega$ satisfies $(PS)_c$ in ${\cal H}$.
\end{definition}

Set $\D_c^G:=B_\alpha({\cal P}^G)\cup B_\alpha(-{\cal P}^G)\cup J^c$ and define
$$
c_j:=\inf\{c\in\R \text{ such that } g(\D^G_c,\D^G_0)\ge j\}.
$$

\begin{proposition}\label{prop}
Assume $J_\Omega$ satisfies $(PS)_{c_j}$ relative to $\D^G_0$ in $H^1_{0,L}(\C)^G$. 
Then, the following holds:
\begin{itemize}
\vspace{-7pt}
\item[(a)] There exists a sign changing critical point $v\in H^1_{0,L}(\Omega)^G$ of $J_\Omega$ with $J_\Omega(v)=c_j$.
\vspace{-7pt}
\item[(b)] If $c_{j+1}=c_j$, then $J_\Omega$ has infinitely many sign changing  critical points $v\in H^1_{0,L}(\C)^G$ with $J_\Omega(v)=c_j$.
\end{itemize}
\vspace{-7pt}
Consequently, if $J_\Omega$ satisfies $(PS)_c$ reltive to $D^G_0$ in $H^1_{0,L}(\C)^G$ for every $c\le d$, then 
$J_\Omega$ has at least $g(\D^G_d,\D^G_0)$ pairs of sign changing critical points $v$ in $H^1_{0,L}(\C)^G$ with $J_\Omega(v)\le d$.
\end{proposition}

Now, we state the mountain pass results for sign changing solutions.

\begin{theorem}\label{mpt}
Let $W$ be a finite dimensional subspace of $H^1_{0,L}(\C)^G$ and let $d:=\sup_W J_\Omega$. 
If $J_\Omega$ satisfies $(PS)_c$ relative to $\D^G_0$ in $H^1_{0,L}(\C)^G$ for every $c\le d$, then 
$J_\Omega$ has at least $\text{dim}(W)-1$ pairs of sign changing critical points $v\in H^1_0(\C)^G$ with 
$J_\Omega(v)\le d$. 
\end{theorem}

For the proofs of  Proposition~\ref{prop} and  Theorem~\ref{mpt}  we refer to 
Proposition 3.6 and Theorem 3.7 in \cite{ClappPacella08}.

%%%%%%%%%%%%%%%%%%%%%%%%%%%%%%%%%%%%%%%%%%%%%
\section{Existence of multiple solutions in annular-shaped domains}

Let $\Gamma$ be a closed subgroup of $O(n)$, and let 
$$
\ell:=\min\{\text{\#}\Gamma x\ : \ x\in\R^n\setminus \{0\}\}
$$

In the proof of our main theorems we need the two following  compactness lemmas. 

\begin{lemma}\label{cpt} 
The energy functional $J_\Omega$ satisfies $(PS)_c$ in $H_{0,L}^1(\C)^\Gamma$ for every $c<\ell c_\infty$.
\end{lemma}

\begin{proof}
Arguing as in the proof of Lemma~\ref{radial} we have that any  $(PS)_c$ sequence $(v_m)_m$ is bounded. 
Thus, 
\begin{eqnarray}
v_m\rightharpoonup v^0&&\text{in }\  H^1_{0,L}(\C),\label{weakv0}\\
u_m\rightharpoonup u^0 &&\text{in }\  L^{2^\sharp}(\Omega),\label{weak2v0}
\end{eqnarray}
where last line follows from the first one and the trace-Sobolev inequality. Hence, 
\begin{equation}\label{weakw}
w_m=v_m-v^0\rightharpoonup 0\quad \text{in } H^1_{0,L}(\C).
\end{equation}
We assume that $w_m\nrightarrow 0$ in $H^1_{0,L}(\C)$, otherwise there is nothing to proof. 

Now we proceed in two steps:

\medskip
\noindent{\bf Step 1} 
{\it The funtion $v^0\in H^1_{0,L}(\C)$ is a weak solution of \eqref{ext}. Moreover, 
\begin{equation*}
J_\Omega(w_m) \to \beta\le c-c_\infty\quad\text{ and}\quad DJ_\Omega(w_m)\to 0.
\end{equation*} 
}
 Indeed, in view of \eqref{weakv0} and \eqref{weak2v0} for any $\varphi\in C^\infty_{0,L}(\C)$ we obtain 
\begin{eqnarray*}
\langle D J_\Omega(v_m),\varphi\rangle 
& = &\int_\C\nabla v_m\cdot\nabla \varphi \ dxdy-\int_\Omega|u_m|^{2^\sharp-2}u_m\varphi \ dx\\
&\to& \int_\C\nabla v^0\cdot\nabla \varphi \ dxdy-\int_\Omega|u^0|^{2^\sharp-2}u^0\ \varphi \ dx
= \langle D J_\Omega(v^0),\varphi\rangle = 0.
\end{eqnarray*}
Hence, $v^0$ weakly solves \eqref{ext}. 

Because of \eqref{weakv0}, \eqref{weak2v0} and  Vitali's theorem (see e.g. \cite[Theorem I.4.2]{Struwe90}) it follows 
\begin{eqnarray*}
\int_\C|\nabla w_m|^2dxdy&=&\int_\C|\nabla v_m|^2dxdy-\int_\C|\nabla v^0|^2dxdy + o(1),\\ 
\int_\Omega|Tr_\Omega(w_m)|^{2^\sharp}dx & =& \int_\Omega|u_m|^{2^\sharp}dx-\int_\Omega|u^0|^{2^\sharp}dx +o(1).
\end{eqnarray*}
where $o(1)\to 0$ as $m\to\infty$. Hence, 
$$
J_\Omega(w_m)=J_\Omega(v_m)-J_\Omega(v^0)+o(1),
$$
and
$$
DJ_\Omega(w_m)=DJ_\Omega(v_m)-DJ_\Omega(v^0)+o(1)=o(1)
$$
where $o(1)\to 0$ in $H^{-1}_{0,L}(\C)$.
Therefore,  
$$
J_\Omega(w_m)\to \beta \le c-c_\infty,\quad DJ_\Omega(w_m)\to 0,
$$
and the claim follows.
\medskip

\noindent{\bf Step 2}
{\it 
Conclusion.
}

Let $v$ be a weak solution of \eqref{ext} in any domain $\C'=\Omega'\times[0,\infty)\subset \R^{n+1}_+$, then 
$$
\int_{\C'}\nabla v\cdot \nabla \varphi\ dxdy-\int_{\Omega'}|u|^{2^\sharp-2}u\varphi\ dx=0
$$
for every $\varphi \in C^\infty_{0,L}(\C')$.
By approximation, we may choose  $\varphi=v$ to get 
$$
0= \langle D J_\Omega(v),\varphi\rangle =\int_\C|\nabla v|^2dxdy -\int_\Omega |u|^{2^\sharp}dx.
$$
Recalling the trace-Sobolev inequality 
\begin{equation}\label{trace_sob}
S_0|u|^{2}_{2^\sharp} \le \Vert v\Vert_2^2=|u|_{2^\sharp}^{2^\sharp},
\end{equation}
we get  that  any non-trivial critical point satisfies 
$$
J_\Omega(v)=\left(\frac{1}{2}-\frac{1}{2^\sharp}\right)|u|_{2^\sharp}^{2^\sharp}
\ge \frac{1}{2n}S_0^\frac{n-1}{2}=c_\infty>0. 
$$

Under our symmetry assumptions, it follows from  Lemma~\ref{lem_struwe} in the Appendix,  
that there must exist zero, or  at least  $\ell$ subsequences $\widetilde{w}_m^j$, 
$j=1,\dots,\ell$, such that,
$$
J_\Omega(v_m) \ge J_\Omega(v^0)+\sum_{j=1}^\ell J_\Omega(w^j_m)+o(1),
$$ 
where $o(1)\to 0$ as $m\to \infty$. Letting $m\to \infty$ we find
$$
c\ge (\ell+1)c_\infty, 
$$
a contradiction with the assumption $c\le \ell c_\infty$. 
\end{proof}

\begin{lemma}\label{cpt2}
If $\ell \ge 2$ the there exist $\varepsilon_0>0$ such that $J_\Omega$ satisfies $(PS)_c$ relative 
to $\D^\Gamma_0$ in $H^1_{0,L}(\C)^\Gamma$ for every $c<(\ell+1)c_\infty+\varepsilon_0$.
\end{lemma}

\begin{proof}
Let $\varepsilon_0\in(0,c_\infty]$ and $(v_m)_m$ be a sequence such that 
$$
v_m\notin \D^\Gamma_0, \quad J_\Omega(v_m)\to c<(\ell+1)c_\infty+\varepsilon_0,\quad DJ_\Omega(v_m)\to 0.
$$
By contradiction assume that  $(v_m)_m$ has no convergent subsequence. We claim that 
there exist $\ell$ sequences $(\widetilde{w}^j_m)_m$ , $j=1,\dots, \ell$ with
\begin{equation}\label{conAp}
\widetilde{w}_m^j(x,y)=(R^j_m)^\frac{n-1}{2}\omega^0(R^j_m(x-x_m^j),R^j_my)
\end{equation}
where $\omega^0\ge 0$ (or $\omega^0\le 0$) is a weak solution of \eqref{ext} in either $\R^{n+1}_{+}$ or in 
$\R^{n+1}_{++}$, and such that 
\begin{equation}\label{conv}
\left\Vert v_m-\sum_{j=1}^\ell \widetilde{w}^j_m\right\Vert \to 0.
\end{equation}
Assuming the claim for a moment we conclude the proof. From \eqref{conv} we  have that 
$$
\text{dist}(v_m,{\cal P}^\Gamma\cup -{\cal P}^\Gamma)\to 0\quad\text{as}\quad m\to \infty,
$$ 
contradicting the fact that $v_m\notin D^\Gamma_0$.
Hence, the conclusion of the lemma follows.

Now we prove our claim. 
By symmetry of the problem, from Lemma~\ref{lem_struwe} in the Appendix,  
there exist at least $\ell$ sequences $(\widetilde{w}^j_m)_m$ such that \eqref{conAp}  holds 
and $\omega^0$ is a weak  solution of \eqref{ext} in  either $\R^{n+1}_{+}$ or in $\R^{n+1}_{++}$.  
Now, to show that that $\omega^0\ge 0$ (or equivalently $\omega^0\le 0)$, we decompose 
the solution $\omega^0$ into its positive and negative parts 
$$
\omega^0=\omega^0_++\omega^0_-,
$$ 
where $\omega_\pm^0=\pm\max\{\pm \omega^0,0\}$. Upon testing \eqref{ext} with $\omega^0_\pm$ 
from the trace-Sobolev inequality \eqref{trace_sob} we find $J_\Omega(\omega^0_\pm)\ge c_\infty$. 
Hence,  
$$
J_\Omega(\omega^0)=J_\Omega(\omega^0_+)+J_\Omega(\omega^0_-) > 2c_\infty.
$$
 Because  invariance under rescaling and symmetry, we have
 $$
 J_\Omega(v_m)\to c \ge \ell J_\Omega(\omega^0) >  2\ell c_\infty\quad\text{as } \ m\to\infty.
 $$
But,  in turn this imply that 
$$
2\ell c_\infty  < (\ell+1)c_\infty+\varepsilon_0 \le (\ell+2)c_\infty.
$$ 
Thus  $\ell < 2$, a contradiction. Hence, either  $\omega_+^0\equiv 0$ or $\omega^0_-\equiv 0$.

Finally, we only need to show \eqref{conv}.  Assume by contradiction that 
$$
v_m-\sum_{j=1}^\ell\widetilde{w}_m^j\nrightarrow 0\quad\text{ in }\ H^1_{0,L}(\C).
$$
Hence, on the one hand side, we may apply Lemma~\ref{lem_struwe} to conclude that there exist a solution $z^0$ of \eqref{ext} and 
a sequence   $\widetilde{z}_m=(R_m)^\frac{n-1}{2}z^0(R_m(x-x_m^j),R_my)$ such that 
\begin{eqnarray*}
J_\Omega\left(v_m-\sum_{j=1}^\ell\widetilde{w}_m^j-\widetilde{z}_m\right)
&= & J_\Omega(v_m)-\ell J_\Omega(\omega^0)- J_\Omega(z^0)+o(1)\\
&<& (\ell+1)c_\infty+\varepsilon_0-(\ell+1) c_\infty<c_\infty.
\end{eqnarray*}
Thus, by Lemma~\ref{lem_struwe}, we find
\begin{equation}\label{onez}
v_m-\sum_{j=1}^\ell\widetilde{w}_m^j-\widetilde{z}_m\to 0\quad\text{ in }\ H^1_{0,L}(\C).
\end{equation}
On the other hand side, because the symmetry assumptions, there must exist $\ell$ sequences 
$(\widetilde{z}_m^j)_m$, $j=1,\dots,\ell$ (all of them generated by $z^0$), such that  
\begin{eqnarray*}
 J_\Omega\left(v_m-\sum_{j=1}^\ell\widetilde{w}_m^j-\sum_{j=1}^\ell\widetilde{z}_m^j\right)
&= & J_\Omega(v_m)-\ell J_\Omega(\omega^0)- \ell J_\Omega(z^0)+o(1)\\
&\le  &c-\ell c_\infty <(2-\ell)c_\infty < c_\infty.
\end{eqnarray*}
Thus,
\begin{equation}\label{ellz}
v_m-\sum_{j=1}^\ell\widetilde{w}_m^j-\sum_{j=1}^\ell\widetilde{z}^j_m\to 0\quad\text{ in }\ H^1_{0,L}(\C).
\end{equation}
Because of the symmetry, combining \eqref{onez} and \eqref{ellz}, we have 
$$
\widetilde{z}_m^j\to 0\quad\text{ in }\ H^1_{0,L}(\C)
$$
for every $j=1,\dots,\ell$,  a contradiction. Hence,  \eqref{conv} and the lemma follows.
\end{proof}

%%%%%%%%%%%%% Poof of main theorems %%%%%%%%%%%%

Now we can give the proofs of our main theorems:

\begin{proof}[Proof of Theorem~\ref{main}]
Let 
$$\ell_0:=\frac{1}{c_\infty}mc(R_1^{1/m},R_2^{1/m})
\quad\text{and}\quad  \omega_1,\dots,\omega_m\in \N(\Omega)^\Gamma$$ 
be positive radial functions as in Lemma~\ref{pos}. 

Let $W_k:=\text{span}\{\omega_1,\dots,\omega_k\}$ be the vector space generated by the first $k$ functions $\omega_1,\dots,\omega_k$. Since for  $i\neq j$ the functions $\omega_i$ and $\omega_j$ have disjoint support, they are orthogonal in $H^1_{0,L}(\C)^\Gamma$. Thus,   $\dim W_k=k$.

Because $\omega_1,\dots,\omega_m\in \N(\Omega)^\Gamma$, for each $k=1,\dots, m$ we have 
$$
\max_{W_k} J_\Omega \le \sum_{i=1}^k\max_{t}J_\Omega(t\omega_i)\le kc(R_1^{1/m},R_2^{1/m}) \le \ell_0c_\infty.
$$
By assumption $\ell > \ell_0 $, then 
$$
\max_{W_k} J_\Omega \le \ell c_\infty 
$$  
and Lemma~\ref{cpt} implies that 
$$
\inf_{\N(\Omega)^\Gamma} J 
$$
is attained at a positive solution $v_1\in \N(\Omega)^\Gamma$ with 
$$
J_\Omega(v_1)\le c(R_1^{1/m},R_2^{1/m}).
$$ 
Moreover, Theorem~\ref{mpt} and Lemma~\ref{cpt2} yield the 
existence of $m-1$ distinct pairs of sign changing critical points 
$\pm v_2,\dots,\pm v_m\in \N(\Omega)^\Gamma$ of $J_\Omega$ with 
$$
J_\Omega(v_k)\le kc(R_1^{1/m},R_2^{1/m}) ,
$$
and this finishes the proof.\end{proof}

\begin{proof}[Proof of Theorem~\ref{main2}]
Let  $\varepsilon_0\in (0,c_\infty)$ be as in Lemma~\ref{cpt2}. Assume without loss of generality that $\delta<\varepsilon_0$. Due to the dilation invariance of $J_\Omega$,  $c(R_1,R_2)=c(R_1/R_2,1)$ and it is easy to see 
that 
$$
c(R,1)\to c_\infty\qquad\text{as }\quad R\to 0.
$$
 Therefore, there exist $R_\delta$ such that 
 $$
 c(R_1^\frac{1}{\ell+1},R_2^\frac{1}{\ell+1})<c_\infty +\frac{\delta}{\ell+1}\qquad\text{ if } \ R_1/R_2<R_\delta.
 $$
As in the proof of Theorem~\ref{main} let $W_k:=\text{span}\{\omega_1,\dots,\omega_k\}$ be the space generated by  $\omega_1,\dots,\omega_k$. As before   $\dim W_k=k$, and for each $k=1,\dots, \ell+1$,   
$$
\max_{W_k} J \le \sum_{i=1}^k\max_{t}J(t\omega_i)\le kc(R_1^\frac{1}{\ell+1},R_2^\frac{1}{\ell+1}) \le 
kc_\infty+\delta<(\ell+1)c_\infty+\varepsilon_0.
$$
Since $\ell\ge 2$, we have that $J_\Omega(\omega_1)\le c_\infty+\delta<\ell c_\infty$.  Thus, on the one hand side, Lemma~\ref{cpt} yields the existence of a positive  solution $v_1\in \N(\Omega)^\Gamma$ with $J_\Omega(v_1) \le c_\infty+\delta$. 

On the other hand side,  Theorem~\ref{mpt} and Lemma~\ref{cpt2} gives the existence of $\ell$ distinct pairs 
of sign changing critical points  $\pm v_2,\dots,\pm v_{\ell+1}\in \N(\Omega)^\Gamma$ of $J_\Omega$ with 
$$
J_\Omega(v_k)\le kc_\infty+\delta, \qquad k=2,\dots,\ell+1,
$$
and this finishes the proof.
\end{proof}

%%%%%%%%%%%%%%%   A P P E N D I X    %%%%%%%%%%%%%
\section{Appendix}

This appendix is devoted to prove Struwe's lemma (see Lemma 3.3 Chapter III in \cite{Struwe90}) in the context of problem \eqref{ext}.  As  mentioned in the introduction, a Liouville theorem in all of $\R^n_+$
for {\it unbounded} solutions of \eqref{half} is still an open question.  Since this theorem is part of the ingredients 
needed in the proof Struwe's lemma, we don't achieve its full strength. Nevertheless, for our purposes the version below 
happens to be enough.      

\begin{lemma}\label{lem_struwe}
Assume $(w_m)_m$ is a $(P.S.)_c$ sequence for $J_\Omega$ in $H^1_{0,L}(\C)$ such that $w_m\rightharpoonup 0$ weakly in $H^1_{0,L}(\C)$ . Then, there exist sequences $(x_m)_m\subset\Omega$, $(R_m)_m$ of raddi 
$R_m\to \infty$ (as $m\to \infty$), a nontrivial solution $\omega^0$ of \eqref{ext} in $\R^{n+1}_+$,  or in $\R^{n+1}_{++}$, and a $(P.S.)_\beta$ sequence  $(\widetilde{w}_m)_m$ for $J_\Omega$ in $H^1_{0,L}(\C)$ such that 
for a subsequence $(w_m)_m$ it holds
\begin{equation}\label{Ap1}
\widetilde{w}_m=w_m-R_m^\frac{n-1}{2}\omega^0(R_m(\cdot -x_m),R_my)\ \rightharpoonup \ 0\quad\text{weakly in } 
\ H^1_{0,L}(\C).
\end{equation}
Furthermore,
\begin{equation}\label{Ap2}
J_\Omega(\widetilde{w}_m)\to \beta=c-J_\Omega(\omega^0)\quad\text{ in }\ H^1_{0,L}(\C).
\end{equation}
Finally, 
\begin{equation}\label{Ap3}
\text{if }\  J_\Omega(w_m)\to c<c_\infty\quad\text{ then }\quad
w_m\to 0 \qquad\text{in}\quad H^1_{0,L}(\C).
\end{equation}

\end{lemma}

\medskip
\begin{proof}
First, we give the argument for  \eqref{Ap3}.  Assuming $J(w_m)\to c<c_\infty$,
because  $w_m$ is a $(PS)_c$ sequence and $w_m\rightharpoonup 0$ weakly, we have
$$
\langle DJ(w_m),w_m\rangle =\Vert w_m\Vert^2 - | w_m|^{2^\sharp}_{2^\sharp} \to 0
$$
Thus, we may assume 
\begin{equation}\label{Ap4}
\Vert w_m\Vert^2\to b\quad\text{and}\quad | w_m|^{2^\sharp}_{2^\sharp}\to b.
\end{equation}
By the trace-Sobolev inequality, we have
$$
\Vert w_m\Vert^2\ge S_0|w_m|^2_{2^\sharp}.
$$
Hence $b\ge S_0b^{2/{2^\sharp}}$ and either $b=0$ or $b\ge S_0^\frac{n-1}{2}$. 
Assuming the latter  $b\ge S_0^\frac{n-1}{2}$, we obtain 
$$
c_\infty=\left(\frac{1}{2}-\frac{1}{2^\sharp}\right)S_0^{\frac{n-1}{2}}\le  \left(\frac{1}{2}-\frac{1}{2^\sharp}\right)b<c_\infty.
$$
a contradiction. Thus $b=0$ and \eqref{Ap3} follows from  \eqref{Ap4}.

Second, we present the argument for \eqref{Ap1} and \eqref{Ap2}.
If $Tr_\Omega(w_m)\to 0$ in $L^{2^\sharp}(\Omega)$, arguing as in the proof of 
 Lemma~\ref{radial} we conclude that  $w_m\to 0$ in $H^1_{0,L}(\C)$, and there is nothing to prove. 
 Hence, we assume $Tr_\Omega(w_m)\nrightarrow 0$ in $L^{2^\sharp}(\Omega)$.  
Thus
$$
\int_\Omega|Tr_\Omega(w_n)|^{2^\sharp}dx\ge \delta \quad \text{ for some }\ \delta \in (0,c_\infty).
$$
 We define the Levy concentration function
$$
Q_m(r):=\sup_{x\in\R^n} \int_{B(x,r)}|Tr_\Omega(w_m)|^{2^\sharp}dx.
$$ 
Since $Q_m(0)=0$ and $Q_m(\infty)>\delta$ there exist a subsequence of $(w_m)_m$, 
sequences $(x_m)_m$ and $(R_m)_m$ such that $x_m\in \Omega$, $R_m>0$ and 
$$
\delta=
\sup_{x\in\R^n}\int_{B(x,R_m)}|T_\Omega(w_m)|^{2^\sharp}dx
=\int_{B(x_m,R_m)}|T_\Omega(w_m)|^{2^\sharp}dx.
$$
In view of $w_m\rightharpoonup 0$  it follows that $R_m\to \infty$ as $m\to \infty$.

Now, letting  
$$
\omega_m(x,y):=(R_m)^\frac{1-n}{2}w_m(x/R_m+x_m,y/R_m),
$$
its is clear, due to the scaling invariance, that 
$$
\delta=
\sup_{x\in\R^n}\int_{B(x,1)}|T_\Omega(\omega_m)|^{2^\sharp}dx
=\int_{B(0,1)}|T_\Omega(\omega_m)|^{2^\sharp}dx.
$$
Moreover, we may assume 
\begin{eqnarray*}
&&\omega_m\rightharpoonup \omega^0\quad\text{in }\ H^1(\R^{n+1}_+),\\
&&\omega_m\to \omega^0\quad \text{a.e. in }\ \R^{n+1}_+.
\end{eqnarray*}
We claim that 
\begin{equation}\label{wnotzero}
\omega^0\neq 0.
\end{equation}

Indeed, let us define $\Omega_m:=\{x\in\R^n\text{ such that }  x/R_m+x_m\in \Omega \}$, $\C_m:=\Omega_m\times [0,\infty)$, and $f_m\in H^1_{0,L}(\C)$ such that 
$$
\langle DJ_\Omega(w_m),h\rangle =\int_\C\nabla f_m\cdot\nabla h\ dxdy\quad \text{for every }\ h\in H^1_{0,L}(\C).
$$
Thus, $g_m(x,y):=R_m^\frac{1-n}{2}\ f_m(x/R_m+x_m,y/R_m)$ satisfies 
\begin{equation}\label{eqg}
\langle DJ_\Omega(w_m),h\rangle =\int_{\C_m}\nabla g_m\cdot\nabla h\ dxdy\quad \text{for every }\ h\in H^1_{0,L}(\C_m).
\end{equation}
Because $w_m\rightharpoonup 0$ in $H^1_{0,L}(\C)$ and the scaling invariance, we get 
\begin{equation}\label{eqng}
\int_{\C_m}|\nabla g_m|^2dxdy=\int_\C|\nabla f_m|^2dxdy=o(1).
\end{equation}

Now, we show \eqref{wnotzero} by contradiction. 
Assume  $\omega^0=0$, thus
\begin{equation}\label{wl2loc}
\omega_n\to 0\text{ in }  L^2_{loc}(\R^{n+1}_+).
\end{equation}
Let $h\in C^\infty_c(\R^{n+1}_+)$ be such that $\text{supp } h \subset B^+((x,y),1)=\{(x,y)\in B((x,y),1)\ \text{ s.t. }y\ge 0\}\subset\R^{n+1}_+$ for some $(x,y)\in \R^{n+1}_+$ and $|h|\le 1$. We find
\begin{eqnarray*}
\int_{\C_m}|\nabla(h\omega_m)|^2dxdy 
& = & \int_{\C_m}\nabla \omega_m\cdot\nabla(h^2\omega_m) dxdy
+\underbrace{\int_{\C_m}|\nabla h|^2|\omega_m|^2dxdy}_{\stackrel{\eqref{wl2loc}}{=}o(1)}\\
& \stackrel{\eqref{eqg}}{\le} & \int_{\C_m}\nabla g_m\cdot \nabla (h\omega_m)h \ dxdy
+\underbrace{\int_{\C_m}\nabla g_m\cdot\nabla h (h\omega_m)dxdy}_{\stackrel{\eqref{eqng}}{=}o(1)}+o(1)\\
&\le&\frac{1}{2}\int_{\C_m}|\nabla(h\omega_m)|^2dxdy +\underbrace{\frac{1}{2}\int_{\C_m}|\nabla g_m|^2dxdy+o(1)}_{\stackrel{\eqref{eqng}}{=}o(1)},
\end{eqnarray*}
where in the last line we have used $|h|\le 1$,  and the Cauchy inequality. Hence, 
\begin{equation}\label{wt0}
\nabla \omega_m\to 0 \quad\text{in }\ L^2_{loc}(\R^{n+1}_+).
\end{equation}
Now, we claim
\begin{equation}\label{ineqtr}
\int_{\R^n}|Tr_\Omega(w)|^{2^\sharp}|Tr_\Omega(\varphi)|^2dx\le 
S_0^{-\frac{1}{2}}\left(\int_{\text{supp}\  w}\hspace{-7pt}|Tr_\Omega(w)|^{2^\sharp}dx\right)^\frac{1}{n}
\hspace{-2pt}\left(\int_{\R^{n+1}_+}\hspace{-4pt}|\nabla(w\varphi)|^2dxdy\right)^2,
\end{equation}
for every $\varphi\in C^\infty_c(\R^{n+1}_+)$, $w\in H^1(\R^{n+1}_+)$.

Indeed, H\"older inequality implies that 
$$
\int_{\R^n}|Tr_\Omega(w)|^{2^\sharp}|Tr_\Omega(\varphi)|^2dx\le 
\left(\int_{\text{supp}\  w}\hspace{-7pt}|Tr_\Omega(w)|^{2^\sharp}dx\right)^\frac{1}{n}
\hspace{-2pt}\left(\int_{\R^{n}}\hspace{-4pt}|Tr_\Omega(w\varphi)|^{2^\sharp}dx\right)^{\frac{n-1}{n}},
$$
and it suffices to use the trace-Sobolev inequality to get \eqref{ineqtr}. 

Combining  \eqref{wt0} and \eqref{ineqtr} we conclude that $Tr_\Omega(\omega_m)\to 0$ in $L^{2^\sharp}_{loc}(\R^n)$. But this is a contradiction with the assumption $\delta>0$. Hence, \eqref{wnotzero} follows. 

Now, we let 
$$
\widetilde{w}_m=w_m-R_m^\frac{n-1}{2}\omega^0(R_m(\cdot-x_m),R_my)
$$
and 
$$
\Omega_\infty=\lim_{m\to\infty} \Omega_m.
$$
By construction $\widetilde{w}_m\rightharpoonup 0$ in $H^1_{0,L}(\C)$, and it is easy to see that 
we have two possibilities: 
$$\Omega_\infty=\R^n\quad\text{ or }\quad  \Omega_\infty=\R^n_+.$$ 
In either case, we may proceed as in the proof of Step 1 to show that $w^0$ is a weak solution of \eqref{ext}  in $\C_\infty=\Omega_\infty\times[0,\infty)$,  
\begin{eqnarray*}
J_\Omega(\widetilde{w}_m)=J_\Omega(w_m ) -J_\Omega(\omega^0) +o(1)\  \to\  
c-J_\Omega(\omega^0),
\end{eqnarray*}
and 
$$
DJ_\Omega(\widetilde{w}_m)\to 0,
$$
as claimed.
\end{proof}

\subsection*{Acknowledgements}
The author wish to thanks Monica Clapp for suggesting the methods and very fruitful discussions.

\end{document}